\documentclass{amsart}
\usepackage{amssymb}
\usepackage{graphicx}
\newtheorem{theorem}{Theorem}[section]

\newtheorem{proposition}[theorem]{Proposition}

\theoremstyle{definition}
\newtheorem{definition}[theorem]{Definition}

\theoremstyle{remark}
\newtheorem{remark}[theorem]{Remark}

\numberwithin{equation}{section}

\begin{document}
\title[Material interpenetration]{Elliptic systems and material interpenetration}
\author{Giovanni Alessandrini}
\address{Dipartimento di Matematica e Informatica, Universit\`a degli Studi di Trieste, Italy}

\email{alessang@units.it}
\thanks{The first author was supported  in part by MiUR, PRIN no. 2006014115}

\author{Vincenzo Nesi}
\address{Dipartimento di Matematica,
La Sapienza, Universit\`a di Roma, Italy}
\email{nesi@mat.uniroma1.it}
\thanks{The second author was supported in part by MiUR, PRIN
no. 2006017833.}


\begin{abstract}
We classify the second order, linear, two by two systems for which the two fundamental theorems for planar harmonic mappings, the Rad\'{o}--Kneser--Choquet Theorem and the H. Lewy Theorem, hold. They are those which, up to a linear change of variable, can be written in diagonal form with \emph{the same} operator on both diagonal blocks.
In particular, we prove that the aforementioned Theorems  cannot be extended to solutions of either the Lam\`{e} system of elasticity, or of elliptic systems in diagonal form, even with just slightly different operators for the two components.
\end{abstract}

\maketitle
\section{Introduction}
\label{intro}

``A basic requirement of continuum mechanics is that interpenetration of matter
does not occur, i.e. that in any deformed configuration the mapping giving the position $u(x)$ of a particle in terms of its position $x$ in the reference configuration be invertible''. J. M. Ball \cite{jmball}.

It is then a natural question to ask what systems of equations, among those used as models for elastostatics, give rise to invertible solutions when reasonable boundary conditions are prescribed.

In this note we shall prove by an example that the Lam\`e system of isotropic, linearized elasticity in the plane, with \emph{constant} Lam\`e coefficients, may lead to physically unacceptable solutions, because interpenetration of matter occurs. Let us recall here that the same phenomenon was previously found by Fosdick and Royer--Carfagni \cite{frc} for a more involved anisotropic linear system, by elaborating on an example due to Lekhnitskii \cite{lek}. In higher dimensions similar phenomena occur. From the mathematical point of view a basic example is due to De Giorgi \cite{degiorgi}. In all these examples, however, a basic common feature is the presence of some sort of point singularity in the solution itself (in dimension greater than two) or at least in the gradient in any dimension.
Such a singularity can only be present when the coefficients of the elliptic system are irregular, for instance discontinuous. Our examples are different from those previously known in several ways, but the most crucial difference is that we choose smooth and in fact \emph{constant} coefficients in our systems. Therefore our examples, besides bringing a new argument to the many already well known, see for instance Ciarlet \cite[p. 286]{ciarlet} and the references therein, about the limitations of linearized elasticity, shed some new light on the tightness of certain classical properties enjoyed by harmonic mapping showing that they cannot be easily extended even within the class of constant coefficient systems.

We now recall some fundamental properties of planar harmonic mappings, see the book of Duren \cite{duren} for a very broad treatment of this subject.

 We begin with the classical theorem of  Rad\'{o}, Kneser and Choquet. This theorem which was first stated by  Rad\'{o} \cite{rado1}, proved by Kneser \cite{kneser} immediately after   and then independently rediscovered by Choquet \cite{Choquet}, remains the basic unequalled result of invertibility for mappings solving an elliptic system of equations.

Let $B$ be an open disk in the plane, let $\Phi:\partial B\to \gamma \subset \mathbb R^2$ be a homeomorphism of $\partial B$ onto a simple closed curve $\gamma$. Let $u\in C^2(B,\mathbb R^2)\cup C(\overline{B},\mathbb R^2)$ be the solution to the Dirichlet problem
\begin{equation}\label{rado}
\left\{
\begin{array}{ccc}
\Delta u=0&,&\hbox{in } B,\\
u=\Phi&,& \hbox{on } \partial B.
\end{array}
\right.
\end{equation}
Let $D$ be the bounded region such that $\partial D=\gamma$. The Rad\'{o}--Kneser--Choquet Theorem states the following.
\begin{theorem}
If $D$ is convex, then $u$ is a homeomorphism of $\overline{B}$ onto $\overline{D}$.
\end{theorem}
The main reasons why this theorem remains substantially unequalled are

\noindent
(i) no analogue of this theorem holds true in dimension three or higher, as it was shown by a striking example by Laugesen \cite{laugesen}, see also Melas \cite{melas},

\noindent
(ii) the convexity assumption on the target domain $D$ is optimal. In fact, it is known since Choquet \cite{Choquet}, that if $D$ is not convex then there exists homeomorphisms $\Phi:\partial B\to \partial D$, for which the injectivity of $u$, the solution to (\ref{rado}), fails. See also \cite{ankneser} for a thorough investigation of this issue.

Nevertheless various kinds of generalizations of the Rad\'{o}--Kneser--Choquet Theorem have been obtained. Regarding harmonic mappings between manifolds, see Schoen and Yau \cite{sy} and Jost \cite{jost}. For mappings $u$ whose components solve a linear elliptic equation, let us mention Bauman, Marini and Nesi \cite{bmn} and also \cite{an:arch}, \cite{an:beltrami}. It is worth pointing out that, in these last two papers the Rad\'{o}--Kneser--Choquet Theorem has been extended to linear elliptic equations in divergence form with merely bounded measurable coefficients.  For quasilinear equations of the $p$-Laplacian type see \cite{as}.

    As a remarkable special case of our  Theorem \ref{generale.th}, which will be stated here below, we prove that no analogue of the Rad\'{o}--Kneser--Choquet Theorem holds when the diagonal Laplacian system is replaced by a Lam\`e system with constant moduli of the following form
    \begin{equation*}
        \mu ~{\rm div} ((\nabla u)^T+ \nabla u) + \lambda ~\nabla ({\rm div} ~u)=0.
    \end{equation*}
More precisely,  we have the following.
\begin{theorem}\label{main.th}
Let  $\mu,\lambda \in \mathbb R$ with $\mu>0$ and $\mu+\lambda>0$. There exist a disk $B\subset \mathbb R^2$, a bounded convex domain $D\subset \mathbb R^2$ and a smooth diffeomorphism $\Phi:\partial B \to \partial D$, so that the unique solution $u\in W^{1,2}(B,\mathbb R^2)$ to
\begin{equation*}
\left\{
\begin{array}{ccc}
    \mu ~{\rm div} ((\nabla u)^T+ \nabla u) + \lambda ~\nabla ({\rm div} ~u)=0&,&\hbox{in }B,\\
    u=\Phi&,&\hbox{on } \partial B,
\end{array}
    \right.
    \end{equation*}
    is \emph{not} a homeomorphism of $B$ onto $D$.
    \end{theorem}

Note that the natural unknown for the Lam\`e system, should be the \emph{displacement} field $\delta$, rather than the \emph{deformation} field $u$. However,  due to the fact the the identity mapping $I$ is also a solution of the Lam\`e system, we have trivially that $\delta$ solves the Lam\`e system if and only $u =I + \delta$ solves the same system. This is the reason why in Theorem \ref{main.th} it is understood that the solution $u$ is representing a deformation field.

    In Remark \ref{more},  we shall see that the solution fails to be an homeomorphism in a very strong way, in fact $u$ maps $B$ onto a domain larger than $D$, moreover it folds itself along a curve, across which the orientation is reversed.

    Our next Theorem is a far reaching generalization of the previous one. It shows that the Rad\'{o}--Kneser--Choquet Theorem holds only if one deals with elliptic systems  \emph{of diagonal form with the same scalar elliptic operator on both components}.
We need some definitions. Consider a constant coefficients second order elliptic system of the form
\begin{equation}\label{gen.syst}
\left\{
\begin{array}{ccc}
  {\rm div}( A \nabla u^1+ B \nabla u^2)=0,      \\
 {\rm div}( C \nabla u^1+ D \nabla u^2)=0.
\end{array}
\right.
\end{equation}
where $A, B, C$ and $D$ are $2\times 2$ real constant matrices,  and the unknowns $u^1$ and $u^2$ are real valued functions.
we say that the system (\ref{gen.syst}) is elliptic if it satisfies the Legendre--Hadamard condition
\begin{equation}\label{ell-1}
\eta_1^2 A \xi\cdot \xi+\eta_1 \eta_2  (B+C) \xi\cdot \xi+ \eta_2^2 D\xi\cdot \xi>0,\quad \hbox{for every\quad} \xi, \eta\in \mathbb R^2\backslash\{0\}.
\end{equation}
This condition is weaker than the strong convexity condition, namely the positivity of the $4\times 4$ matrix given in block form as
\begin{equation*}
M= \left(
\begin{array}{ccc}
A& B      \\
C&D
\end{array}
\right)
\end{equation*}
and, as it is well known, ellipticity \eqref{ell-1} is the same as rank one convexity of the quadratic form associated to $M$, see \cite[Theorem 5.3]{dacorogna}.
We note that there is no loss of generality in assuming that the matrices $A,B,C$ and $D$ are symmetric.
\begin{definition}
We shall say that the system (\ref{gen.syst}) is \emph{equivalent} to the system
\begin{equation*}
\left\{
\begin{array}{ccc}
  {\rm div}( A^{\prime} \nabla u^1+ B^{\prime} \nabla u^2)=0,      \\
 {\rm div}( C^{\prime} \nabla u^1+ D^{\prime} \nabla u^2)=0
\end{array}
\right.
\end{equation*}
if there exists a non-singular $2\times 2$ matrix
$\left(
\begin{array}{cc}
  \alpha&  \beta   \\
  \gamma&   \delta\end{array}
\right)$
such that
\begin{equation}\label{equiv}
\left(
\begin{array}{ccc}
A& B      \\
C&D
\end{array}
\right)=
\left(
\begin{array}{cc}
  \alpha {\rm Id}&  \beta {\rm Id}  \\
  \gamma {\rm Id}&   \delta {\rm Id}
\end{array}
\right)
\left(
\begin{array}{cc}
A^{\prime}& B^{\prime}      \\
C^{\prime}&D^{\prime}
\end{array}
\right).
\end{equation}
\end{definition}
\begin{theorem}\label{generale.th}
Let the ellipticity condition (\ref{ell-1}) be satisfied.
The following alternative holds.

\noindent
Either

\noindent
i) the system (\ref{gen.syst}) is equivalent to the system
\begin{equation}\label{diag}
\left\{
\begin{array}{ccc}
  {\rm div}( A \nabla u^1)=0,      \\
 {\rm div}(A\nabla u^2)=0,
\end{array}
\right.
\end{equation}

\noindent
or

\noindent
ii)
     there exist a disk $B\subset \mathbb R^2$, a bounded convex domain $D\subset \mathbb R^2$  and a smooth diffeomorphism $\Phi:\partial B \to \partial D$, so that the unique solution $u=(u^1,u^2)\in W^{1,2}(B,\mathbb R^2)$ to
     \begin{equation*}
\left\{
\begin{array}{rcl}
  {\rm div}( A \nabla u^1+ B \nabla u^2)=0&,&\hbox{in $B$},      \\
 {\rm div}( C \nabla u^1+ D \nabla u^2)=0&,&\hbox{on $B$} ,\\
u=\Phi&,& \hbox{on } \partial B.
\end{array}
\right.
\end{equation*}
is \emph{not} a homeomorphism of $B$ onto $D$.
    \end{theorem}
    The above results show that one of the most basic properties enjoyed by planar harmonic mappings cannot be extended to other elliptic systems in the plane. It is then natural to ask similar questions for another fundamental property of injective harmonic mapping. A benchmark of the theory is a result of H. Lewy \cite{lewy} proving that harmonic homeomorphisms are, in fact, diffeomorphisms. More precisely we have the following result.
    \begin{theorem}\label{lewy.th}(H. Lewy.)
    Lut $u=(u^1,u^2) :B\to \mathbb R^2$ be a harmonic mapping. If
    $u$ is invertible, then
    \begin{equation}
    \det Du\neq 0\quad\hbox{for every \quad} (x,y)\in B.
    \end{equation}
    \end{theorem}
    Also in this case the validity is limited to two dimensions. J. C. Wood \cite{wood2} found a  third degree polynomial harmonic mapping which provides a counterexample in dimension three. On the positive side, Hans Lewy \cite{lewy2} recognized that, in three dimensions, if $u$ is the \emph{gradient} of an harmonic function and it is a homeomorphism, then it is a diffeomorphism. This result was extended to any dimension in the remarkable paper by Gleason and Wolff \cite{gleasonwolff}.
In a different direction, several generalizations of Lewy's
Theorem have been achieved in dimension two when the components of
$u$ satisfy the \emph{same} linear elliptic equation of the form
${\rm div} (\sigma \nabla u^i)=0$. For the case of sufficiently
smooth $\sigma$ see \cite{bmn}. When $\sigma$ is allowed to be
discontinuous, weak forms of Lewy's Theorem have been obtained in
\cite{an:arch} and \cite{an:beltrami}. A version for $p$-Laplacian
type equations can be found in \cite{as}.

In the next theorem, we show, by means of examples, that Theorem
\ref{lewy.th} cannot  be extended to an arbitrary elliptic system
with constant coefficients, unless, again,  the systems has the
special  form \eqref{diag}.

\begin{theorem}\label{no-HL.th}
Let the ellipticity condition (\ref{ell-1}) be satisfied.
The following alternative holds.

\noindent
Either

\noindent
i) the system (\ref{gen.syst}) is equivalent to the system \eqref{diag}

\noindent
or

\noindent
ii)
     there exists a polynomial solution  to
     \begin{equation*}
\left\{
\begin{array}{rcl}
  {\rm div}( A \nabla u^1+ B \nabla u^2)=0&,&\hbox{in $B$},      \\
 {\rm div}( C \nabla u^1+ D \nabla u^2)=0&,&\hbox{on $B$} ,\end{array}
\right.
\end{equation*}
which is a homeomorphism of a closed disk $\overline{B}$ onto $u(\overline{B})$ and such that in the center of the disk, denoted by $O$, we have
\begin{equation*}
\det Du(O)=0.
\end{equation*}
    \end{theorem}
     \begin{remark}
    It may be evident that the Rad\'{o}--Kneser--Choquet and the H. Lewy Theorems continue to hold for any system of the form (\ref{diag}),
    since it can be elementarily reduced to a Laplacian diagonal system via a linear change of the independent coordinates.
    It is although rather remarkable that Theorems \ref{generale.th}, \ref{no-HL.th} show that the Rad\'{o}--Kneser--Choquet and the H. Lewy Theorems do
    not extend to very slight perturbations of the Laplacian diagonal system such as, for instance, the following one
\begin{equation*}
\left\{
\begin{array}{ccc}
    u^1_{xx}+ u^1_{yy}=0,\\
    (1+\varepsilon)u^2_{xx}+ u^2_{yy}=0,\
\end{array}
    \right.
    \end{equation*}
where $\varepsilon$ is any positive number.
\end{remark}

    \section{Proofs}
In what follows, when no ambiguity occurs, we shall identify points $(x,y)\in \mathbb R^2$ with column vectors
$\left(
    \begin{array}{c}
    x\\
    y
    \end{array}
    \right)$.
    Also, for $\theta \in \mathbb R$, we shall denote  $c_{\theta} = \cos\theta$ ,  $s_{\theta} = \sin \theta$.
For the proofs of Theorems  \ref{generale.th} and \ref{no-HL.th} we shall make use of the following two propositions, which we will prove at the end of this Section.
\begin{proposition}\label{A.prop}
Let the ellipticity condition (\ref{ell-1}) be satisfied.
If the system (\ref{gen.syst}) is \emph{not} equivalent to
(\ref{diag}), then there exists $\theta \in [0,2 \pi]$ and a
quadratic polynomial $p(x,y)= \frac{1}{2}(a x^2+ 2 b x y + c y^2)$
such that
\begin{equation*}
\left(
\begin{array}{c}
  u^1   \\
 u^2
\end{array}
\right)
=
\left(
\begin{array}{cc}
  c_{\theta}  & -s_{\theta}   \\
  s_{\theta}&   c_{\theta}
\end{array}
\right)
\left(
\begin{array}{ccc}
x^2+y^2  \\
p(x,y)
\end{array}
\right)
\end{equation*}
is a solution to (\ref{gen.syst}).
\end{proposition}
\begin{proposition}\label{B.prop}
Let the ellipticity condition (\ref{ell-1}) be satisfied.
If the system (\ref{gen.syst}) is \emph{not} equivalent to
(\ref{diag}), then there exists $\theta \in [0,2 \pi]$ and a cubic
polynomial $$q(x,y)= \frac{1}{2}\left(a \frac{x^3}{3}+  b x^2 y +
c x y^2 +d \frac{y^3}{3}\right)$$ such that
\begin{equation*}
\left(
\begin{array}{c}
  u^1   \\
 u^2
\end{array}
\right)
=
\left(
\begin{array}{cc}
  c_{\theta}  & -s_{\theta}   \\
  s_{\theta}&   c_{\theta}
\end{array}
\right)
\left(
\begin{array}{ccc}
x(x^2+y^2)  \\
q(x,y)
\end{array}
\right)
\end{equation*}
is a solution to (\ref{gen.syst}).
\end{proposition}

    \begin{proof}[Proof of Theorem \ref{generale.th}.]
    We assume that (\ref{gen.syst}) is \emph{not} equivalent to (\ref{diag}). We choose $\theta$ according to Proposition \ref{A.prop} and we write
    \begin{equation*}
    R_{\theta}=
\left(
\begin{array}{cc}
  c_{\theta}  & -s_{\theta}   \\
  s_{\theta}&   c_{\theta}
\end{array}
\right).
\end{equation*}
    Being linear mappings solutions to  (\ref{gen.syst}), we have that also the following is solution to  (\ref{gen.syst})
    \begin{equation*}
    \left(
\begin{array}{c}
  u^1   \\
 u^2
\end{array}
\right)
=
R_{\theta}
\left(
\begin{array}{ccc}
x^2+y^2 -1 \\
k y +p(x,y)
\end{array}
\right)
\end{equation*}
where $k\neq 0$ is a constant to be determined later on.

    We choose $B=\left\{(x,y)\in \mathbb R^2 : \left(x-\frac 1 2\right)^2 + y^2<\frac 5 4\right\}$. We have
\begin{equation*}
    x=-1 + x^2+y^2\qquad\hbox{on}\quad\partial B.
    \end{equation*}
    We now select $\Phi$. We set
    \begin{equation*}
    \begin{array}{ccc}
    M=  \left(
    \begin{array}{cc}
    1&0\\
    0&k
    \end{array}\right),
&\Psi(x,y)=
    \left(
    \begin{array}{c}
    x\\
    y + \frac{1}{k} p(x, y)
    \end{array}
    \right),& \Phi(x,y)=R_{\theta}
 M\Psi(x,y).
    \end{array}
    \end{equation*}
    A straightforward calculation shows that when
    \begin{equation}\label{k}
    k\geq |b|
    \end{equation}
    $\Psi$ is a homeomorphism of $\partial B$ onto a closed convex curve $\Gamma$. Consequently $\Phi$ is also a homeomorphism of $\partial B$ onto the closed convex curve $\gamma=R_{\theta}
M \Gamma$.

    Let $D$ be the bounded convex domain such that $\partial D =\gamma$. It is easy to check that $\overline{D}\subset R_{\theta} M
S$ where
    \begin{equation*}
    S=\left\{(v^1,v^2)\in \mathbb R^2: \frac{1-\sqrt{5}}{2}\leq v^1\leq \frac{1+\sqrt{5}}{2}\right\}.
    \end{equation*}
    However
    \begin{equation*}
     u(0,0)=-R_{\theta}M\left(
    \begin{array}{c}    1\\0
    \end{array}
    \right)\notin R_{\theta} M S.
    \end{equation*}
\end{proof}
    \begin{remark}
    \label{more}
    Note that we can compute the the Jacobian determinant of $u$ in terms of the coefficients of $p$ and obtain $\det Du = 2 k x + 2(b(x^2-y^2)+(c-a) x y)$. Hence the Jacobian determinant vanishes at $(0,0)$ and, in fact, it changes sign across its \emph{nodal line}
    \begin{equation*}
    H=\left\{(x,y)\in \mathbb R^2:  k x + (b(x^2-y^2)+(c-a) x y)=0\right\}.
    \end{equation*}
    which is always an hyperbola (unless $a=b=c=0$ when it degenerates in a straight line). Note that $k x + (b(x^2-y^2)+(c-a) x y)$ is positive in $(1,0)$ zero in $(0,0)$ and negative in $(-1,0)$ because of (\ref{k}). See Figures $1$ and $2$
where the behaviour of $u$ and its Jacobian determinant are depicted in the specific case of the Lam\`e system.
    \end{remark}

\begin{proof}[Proof of Theorem \ref{no-HL.th}.]
We assume again that (\ref{gen.syst}) is \emph{not} equivalent to (\ref{diag}). We choose $\theta$ and $q$ according to Proposition \ref{B.prop}. Being linear mappings solutions to  (\ref{gen.syst}), we have that also the following is solution to  (\ref{gen.syst})
    \begin{equation*}
        \left(
\begin{array}{c}
  u^1   \\
 u^2
\end{array}
\right)
=
\left(
\begin{array}{cc}
  c_{\theta}  & -s_{\theta}   \\
  s_{\theta}&   c_{\theta}
\end{array}
\right)
\left(
\begin{array}{ccc}
x(x^2+y^2) \\
y +q(x,y)
\end{array}
\right).
\end{equation*}
It is now easy to check that, once we have chosen the polynomial $q$ according to Proposition \ref{B.prop}, the following two properties hold. First there exists a positive radius $r$ such that
one has $\det Du(x,y)>0$ if $(x,y)\in B_r(O)\backslash \{O\}$, where we have set $O=(0,0)$, and second  $\det Du(O)=0$.
Choose $0<\rho<r$, and denote by
\begin{equation*}
\Phi=u \big|_{\partial B_{\rho}(O)}.
\end{equation*}
A very simple calculation shows that $\Phi$ maps $\partial B$ in a one to one way onto a closed curve $\gamma$ provided one has
\begin{equation*}
2+b \rho^2>0\quad\hbox{and \quad} 3+d \rho^2>0.
\end{equation*}
Let $D$ be the bounded domain such that $\partial D=\gamma$. We can now apply a topological result of Meisters and Olech \cite{meistol}. Indeed we have a smooth mapping $u$ defined on a closed disk $\overline{B}$ and which is a local homeomorphism at each point of $\overline{B}$, with the possible exception of the point $O$ only.  Moreover, the restriction of $u$ to $\partial B$ is a homeomorphism. Therefore the hypotheses of  Theorem 1 in \cite{meistol}  (see also Corollary 2) are satisfied and we can conclude  that $u$ is a homeomorphism of $\overline{B_{\rho}(O)}$ onto $\overline{D}$.
\end{proof}
\begin{proof}[Proof of Theorem \ref{main.th}.]
It suffices to verify that the Lam\`e system  is not equivalent to any elliptic system of the form (\ref{diag}). In fact it can be rewritten in the form (\ref{gen.syst}) with the
following choices
\begin{equation*}
\begin{array}{ccccc}
A=\left(
\begin{array}{cc}
 2 \mu +\lambda&0\\
 0&\mu
  \end{array}
\right),
&
B=C\left(
\begin{array}{cc}
 0 &\frac{ \mu +\lambda}{2}\\
 \frac{ \mu +\lambda}{2}&0     \end{array}
\right),
&
D=\left(
\begin{array}{cc}
\mu&0\\
 0&     2 \mu +\lambda
 \end{array}
\right).
\end{array}
\end{equation*}
Now $B, C$ and $D$ can be scalar multiples of $A$ only if $\mu+\lambda=0$, which contradicts the assumption $\mu+\lambda>0$.
\end{proof}
\begin{remark}
Note that our assumptions $\mu>0$, $\mu + \lambda > 0$ correspond the the strong convexity assumption, which, as is well known, is stronger that the ellipticity condition \eqref{ell-1}. \end{remark}

    In the following picture we illustrate our results for the case of the Lam\`e system. We choose
    $\lambda=\mu
=1$. With this choice, condition \eqref{k} takes the form $k\geq(1+\sqrt{10})\frac{\lambda+ 3 \mu}{\lambda+\mu}$ and, for the picture we have chosen the limiting value $k=2(1+\sqrt{10})$.
Note that, in order to facilitate visibility, the coordinates in the $u^1$ and the $u^2$ directions are scaled differently.

\begin{figure}[h]
\label{gamma}
\centerline{
\includegraphics[scale=.55,angle=0]{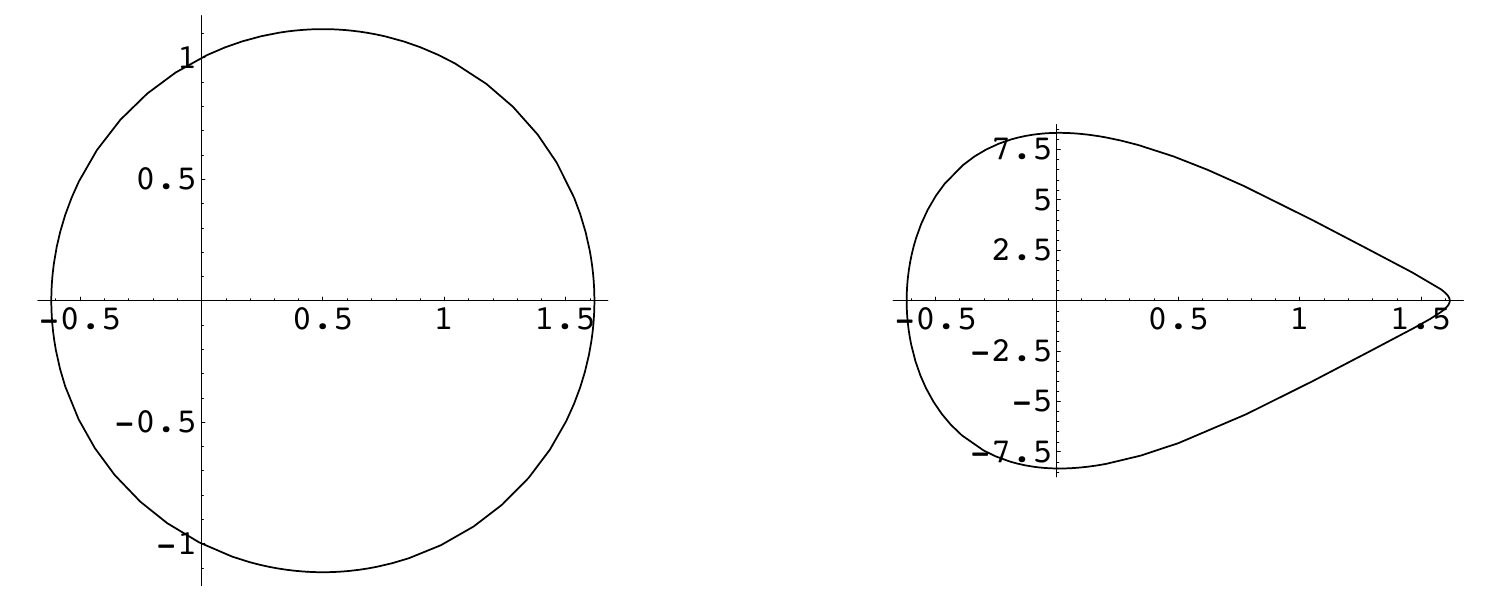}
}
\caption{$\partial B$ and its image $\Phi(\partial B)$.}
\end{figure}

\begin{figure}[h] \label{nodal}
\centerline{
\includegraphics[scale=1,angle=0]{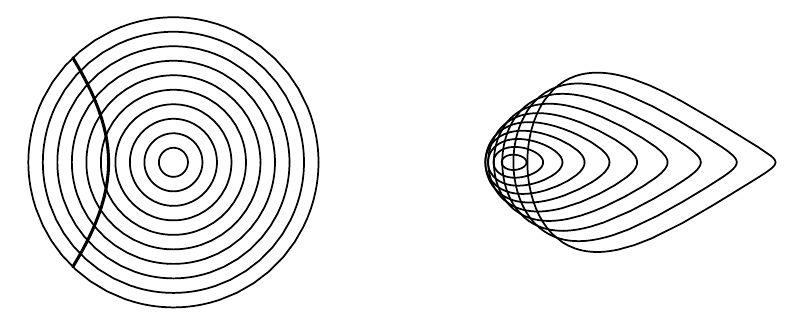}
}
\caption{Left: circles $C_r$ of varying radii and the nodal  line of the Jacobian (an hyperbola) drawn within $B$. Right:   the images $U(C_r)$.}
\end{figure}

\begin{proof}[Proof of Proposition \ref{A.prop}]
We look for $\theta\in [0,2 \pi]$
and a quadratic polynomial $p(x,y)= \frac{1}{2}(a x^2+ 2 b x y + c y^2)$ such that
\begin{equation}\label{quad.sol}
\left(
\begin{array}{c}
  u^1   \\
 u^2
\end{array}
\right)
=
\left(
\begin{array}{cc}
  c_{\theta}  & -s_{\theta}   \\
  s_{\theta}&   c_{\theta}
\end{array}
\right)
\left(
\begin{array}{ccc}
x^2+y^2  \\
p(x,y)
\end{array}
\right)
\end{equation}
is a solution to (\ref{gen.syst}).

Since the Hessian matrices of $u^i$ are constant, the system (\ref{gen.syst}) is equivalent to the following two equations
\begin{equation*}
\left(-s_{\theta} F+c_{\theta} G\right) \left(\begin{array}{c}
 a\\b\\c
  \end{array}
 \right)= Y
\end{equation*}
where
\begin{equation*}
\begin{array}{cccc}
F=\left(\begin{array}{ccc}
 a_{11}&2a_{12}& a_{22}\\
  c_{11}&2c_{12}& c_{22}
 \end{array}
 \right),
 &
G=\left(\begin{array}{ccc}
 b_{11}&2b_{12}& b_{22}\\
  d_{11}&2d_{12}& d_{22}
 \end{array}
 \right),
 \end{array}
\end{equation*}
and $ Y=\left(\begin{array}{c}
 Y^1\\
 Y^2
 \end{array}
 \right)$ is a vector of known data, possibly depending on $\theta$. Given $\theta$,
the above system has at least one solution if  the rank of  $-s_{\theta} F+c_{\theta} G$ is two and consequently we find a solution of the form \eqref{quad.sol} to \eqref{gen.syst}. If this were not the case, then for every $\theta$ there exists $\phi$ such that
\begin{equation}\label{linear.dip}
c_{\phi} (-s_{\theta} A +c_{\theta} B)+ s_{\phi}(-s_{\theta} C +c_{\theta} D)=0.
\end{equation}
Choosing $\theta=0$ yields that $B$ and $D$ are linearly
dependent.  Similarly, by choosing $\theta=\frac{\pi}{2}$, we get
that $A$ and $C$ are linearly dependent. Recalling that, by
ellipticity \eqref{ell-1}, $A$ and $D$ are positive definite, and
hence nontrivial, we obtain  that $B= \sigma D$ and $C= \gamma A$
for suitable constants $\sigma, \gamma \in \mathbb{R}$. Plugging
these linear dependencies into \eqref{linear.dip} and using once
more ellipticity, we obtain that also
 $D$ is a scalar multiple of $A$. In conclusion $B, C$ and $D$ are scalar multiples of $A$.

\end{proof}
\begin{proof}[Proof of Proposition \ref{B.prop}]
We look for  $\theta \in [0,2 \pi]$ and a cubic polynomial $$q(x,y)= \frac{1}{2}\left(a \frac{x^3}{3}+  b x^2 y + c x y^2 +d \frac{y^3}{3}\right)$$ such that
\begin{equation*}
\left(
\begin{array}{c}
  u^1   \\
 u^2
\end{array}
\right)
=
\left(
\begin{array}{cc}
  c_{\theta}  & -s_{\theta}   \\
  s_{\theta}&   c_{\theta}
\end{array}
\right)
\left(
\begin{array}{ccc}
x(x^2+y^2)  \\
q(x,y)
\end{array}
\right)
\end{equation*}
is a solution to (\ref{gen.syst}).
This time the Hessian matrices of $u^i$ are of the form $x H_1+y H_2$ for suitable constant matrices $H_1$ and $H_2$ and thus (\ref{gen.syst}) imposes the following four conditions.
\begin{equation*}
\left(-s_{\theta} F+c_{\theta} G\right) \left(\begin{array}{c}
 a\\b\\c\\d
  \end{array}
 \right) = Y
\end{equation*}
where
\begin{equation*}
\begin{array}{cccc}
F=\left(\begin{array}{cccc}
 a_{11}&2a_{12}& a_{22}&0\\
 0& a_{11}&2a_{12}& a_{22}\\
  c_{11}&2c_{12}& c_{22}&0\\
0& c_{11}&2c_{12}& c_{22}
 \end{array}
 \right),
 &
G=\left(\begin{array}{cccc}
 b_{11}&2b_{12}& b_{22}&0\\
  0&b_{11}&2b_{12}& b_{22}\\
   d_{11}&2d_{12}& d_{22}&0\\
  0&d_{11}&2d_{12}& d_{22}
 \end{array}
 \right),
 \end{array}
\end{equation*}
and $ Y\in \mathbb R^4 $ is a data vector.
We make use of the following linear algebra fact.
 Given any two $2\times 2$ symmetric
matrices
\begin{equation*}
M=\left(\begin{array}{cccc}
 m_{11}&m_{12}\\
m_{12}& m_{22}
 \end{array}
 \right)\hbox{and\quad}
S= \left(\begin{array}{cccc}
 s_{11}&s_{12}\\
s_{12}& s_{22}
 \end{array}
 \right),
 \end{equation*}
if  we have
 \begin{equation*}
 \det \left(\begin{array}{cccc}
 m_{11}&2m_{12}& m_{22}&0\\
 0& m_{11}&2m_{12}& m_{22}\\
  s_{11}&2s_{12}& s_{22}&0\\
0& s_{11}&2s_{12}& s_{22}\\
 \end{array}
 \right)=0,
\end{equation*}
and either $M$ or $S$ is positive definite, then $M$ and $S$ are
linearly dependent. This fact may be verified in many ways, for
instance with the aid of Gaussian elimination.
Assume that $(-s_{\theta} F+ c_{\theta} G)$ is singular for every $\theta$, then we must have that the matrix $(-s_{\theta} C+ c_{\theta} D)$
 is a scalar multiple of  $
(-s_{\theta} A+ c_{\theta} B) $ at least for all those   $\theta$
for which $ (-s_{\theta} A+ c_{\theta} B) $ is positive definite.
Equivalently, switching the roles of $A, B$ and $D, C$, if
$(-s_{\theta} C+ c_{\theta} D)$ is positive definite, then  $
(-s_{\theta} A+ c_{\theta} B) $  is a scalar multiple of
$(-s_{\theta} C+ c_{\theta} D)$. Recalling that the matrices $A,
D$ are positive definite by ellipticity, and choosing $\theta=0,
\frac{3\pi}{2}$, we deduce $B= \sigma D$ and $C= \gamma A$ for
suitable constants $\sigma, \gamma \in \mathbb{R}$. Moreover, it
is evident that there exists an open interval $I$ containing
$\frac{3\pi}{2}$ for which $(-s_{\theta} A+ c_{\theta} B) $
remains positive definite as long as $\theta \in I$. Thus, for all
$\theta\in I$, there exists $\phi$ such that \eqref{linear.dip}
holds true, and from now
 we  can argue similarly as in the proof of Proposition \ref{A.prop}.
\end{proof}

\providecommand{\bysame}{\leavevmode\hbox to3em{\hrulefill}\thinspace}
\providecommand{\MR}{\relax\ifhmode\unskip\space\fi MR }
\providecommand{\MRhref}[2]{%
  \href{http://www.ams.org/mathscinet-getitem?mr=#1}{#2}
}
\providecommand{\href}[2]{#2}

\end{document}